\documentclass[11pt]{amsart}

\usepackage{amsfonts, amstext, amsmath, amsthm, amscd, amssymb} 
\usepackage{mathtools}

\usepackage{float} 
\usepackage{graphics}
\usepackage{caption}
\usepackage{subcaption}
\usepackage{import}

\usepackage{physics}  
\usepackage{microtype}

\usepackage[hidelinks,pagebackref,pdftex]{hyperref}

\usepackage[dvipsnames]{xcolor}
\definecolor{MyCyan}{HTML}{00F9DE}
\usepackage{enumitem} 
\setlist{nolistsep,leftmargin=*}
\usepackage{transparent} 
\usepackage{mathrsfs}

\usepackage{marginnote}
\long\def\@savemarbox#1#2{\global\setbox#1\vtop{\hsize\marginparwidth 
  \@parboxrestore\tiny\raggedright #2}}
\newcommand{\ZZ}{\mathbb{Z}}  % Integers
\newcommand{\calP}{\mathcal{P}}
\newcommand{\calL}{\mathcal{L}}

\renewcommand{\setminus}{{\smallsetminus}}

\newcommand{\from}{\colon\thinspace} % as in f \from X \to Y

\newtheorem{theorem}{Theorem}[section]
\newtheorem*{theorem*}{Theorem}
\newtheorem{proposition}[theorem]{Proposition}
\newtheorem{lemma}[theorem]{Lemma}
\newtheorem{corollary}[theorem]{Corollary}

\newtheorem*{namedtheorem}{\theoremname}
\newcommand{\theoremname}{testing}
\newenvironment{named}[1]{\renewcommand{\theoremname}{#1}\begin{namedtheorem}}{\end{namedtheorem}}

\theoremstyle{definition}
\newtheorem{definition}[theorem]{Definition}

\newcommand{\refthm}[1]{Theorem~\ref{Thm:#1}}

\newcommand{\refprop}[1]{Proposition~\ref{Prop:#1}}
\newcommand{\refcor}[1]{Corollary~\ref{Cor:#1}}

\newcommand{\refdef}[1]{Definition~\ref{Def:#1}}
\newcommand{\refsec}[1]{Section~\ref{Sec:#1}}
\newcommand{\reffig}[1]{Figure~\ref{Fig:#1}}

\title{The Lorenz braid index and hyperbolic volume} 

\author{Thiago de Paiva}
\address[]{Beijing International Center for Mathematical Research, Peking University, Beijing 100871, China P.R.}
\email[]{thhiagodepaiva@gmail.com}

\author[Connie Hui]{Connie On Yu Hui} 
\address[]{Sydney Mathematical Research Institute (SMRI), The University of Sydney, NSW 2006, Australia } 
\email[]{connieonyuhui@gmail.com | connie.hui@sydney.edu.au} 

\author{Jos\'{e} Andr\'{e}s Rodr\'{\i}guez Migueles}
\address[]{Centro de Investigaci\'on en Matem\'aticas, GTO 36023, Mexico}
\email[] {jose.migueles@cimat.mx} 

\subjclass[2020]{57K10, 57K32 (primary); 37D40, 37E35 (secondary).}  

\begin{document} 
\begin{abstract}
A result of Futer, Kalfagianni, and Purcell implies that an upper volume bound for all link complements in the 3-sphere cannot depend solely on the braid index. In this paper, we introduce the Lorenz braid index and generalise the bunch algorithm to provide a general upper volume bound for all link complements in the 3-sphere. Such an upper bound is a quadratic polynomial in the Lorenz braid index. In addition, we construct an explicit family of hyperbolic Lorenz knots for which the classical braid index and the Seifert genus both tend
to infinity, while the Lorenz braid index remains bounded.
\end{abstract} 

\maketitle 
 
\section{Introduction}
A central theme in 3-manifold topology is the search for relationships between topological and geometric invariants of
3-manifolds such as knot or link complements. For a hyperbolic link \(L \subset S^3\),
a fundamental
geometric invariant is the hyperbolic volume of its complement. On the other
hand, many natural measures of link complexity are diagrammatic or topological, e.g. crossing number, twist number, bridge number, braid index, and genus. A widely-studied problem is to find out which of these quantities estimate hyperbolic volume.  

Some known upper or lower bounds for hyperbolic volumes depend only on the twist number of a given link diagram. For instance, Agol and D.~Thurston showed that the volume of a hyperbolic 
link can be bounded above linearly in terms of the twist number of a diagram, by improving a result of  Lackenby~\cite{Lackenby}. Moreover, Futer, Kalfagianni, and  Purcell gave a lower bound linearly in terms of the twist number of a prime, twist–reduced diagram for any hyperbolic link with at least 7 crossings in each twist region \cite{FKPDehnFillingJones}. Nevertheless, for general knots, the twist number does not provide two–sided bounds on the volume \cite{FKPDiagrammaticBounds}. Related diagrammatic estimates have been obtained for several structured families of links, including highly twisted links~\cite{PurcellHighlyTwisted}. Such bounds, however,
depend on a chosen diagram. 
By contrast, the braid index is an intrinsic invariant of the link. It is therefore natural to ask whether braid index carries meaningful geometric information.

The work of Purcell and Zupan~\cite{PurcellZupan} shows that classical
topological complexity need not force hyperbolic volume to grow. More
precisely, for every \(g\geq0\), they constructed a sequence of knots with
uniformly bounded volumes and genus-\(g\) bridge number tending to infinity.
Note that genus-$0$ bridge number is the classical bridge number, and the classical bridge number is always smaller than or equal to the braid index of the same link (consider the minimal braid diagram of a link, the number of bridges is smaller than or equal to the number of braids). Thus, the classical braid indices of these examples also tend to infinity,
while their volumes remain uniformly bounded. 

On the other hand, the general upper bound for volumes of all hyperbolic link complements cannot depend solely on the classical braid index because there exist a sequence of hyperbolic links with braid index equal to three such that their hyperbolic volumes grow to infinity~(see~\cite[Theorem~5.6]{Futer-Kalfagianni-Purcell:CuspAreasFareyMfds} and \cite[Theorem~1]{Birman-Menasco:ANoteOnClosed3Braids}). 

Thus, the classical braid index does not always give good estimate for hyperbolic volumes:
large braid index does not necessarily give us
large volume; and large volume does not necessarily give us large braid index. However, this does not preclude the possibility that other braid-theoretic quantities can give us good estimate.

In this paper, we introduce the \textit{Lorenz braid index}, which can provide a general upper bound for the volumes of (the hyperbolic pieces of) each link complements in the $3$-sphere~$S^3$. 

Recall that the braid group
\(B_m\) is generated by the standard generators
\[
\sigma_1,\ldots,\sigma_{m-1},
\]
where \(\sigma_i\) represents a positive crossing between the \(i\)-th and
\((i+1)\)-th strands. A braid is positive if it is represented by a word in
these generators with only positive exponents. By the theorem of Birman and
Kofman~\cite{Birman-Kofman:NewTwistOnLorenzLinks}, Lorenz links are precisely T-links, that is,
closures of positive braids of the form
\[
(\sigma_1\cdots\sigma_{r_1-1})^{s_1}
\cdots
(\sigma_1\cdots\sigma_{r_k-1})^{s_k},
\]
where
\[
2\leq r_1<\cdots<r_k
\qquad\text{and}\qquad
s_1,\ldots,s_k>0.
\]

Generalised T-links extend this class by allowing the last exponent to be negative or positive, together with a possible
shift of the strands. They form a universal class: every link in \(S^3\)
admits a representation as a generalised T-link~\cite{dePaiva-Hui-Rodriguez:GeneralisedTLinks}. Such a representation
has the form
\[
L=
T((r_1,s_1),\ldots,(r_k,s_k),(r_{k+1},s_{k+1}),d).
\]
For any integer \(n\) satisfying
\[
s_{k+1}\geq nr_{k+1},
\]
we remove \(n\) full twists from the final block and obtain the Lorenz link
\[
L_n^+
=
T((r_1,s_1),\ldots,(r_k,s_k),
(r_{k+1},s_{k+1}-nr_{k+1})),
\]
where the final block is omitted if \(s_{k+1}-nr_{k+1}=0\). We call $L_n^+$ an 
\textit{associated Lorenz link} of the chosen generalised T-link
representation $L$. 

We denote by \(\operatorname{br}(J)\) the braid index of a link \(J\), that
is, the minimum number of strands among all closed braid representatives of
\(J\). For an arbitrary link this invariant is often difficult to compute. However, for Lorenz links, the braid index is the same as its trip number in the classical Lorenz template (see \cite{Franks-Williams:BraidsAndJonesPolynomial, Birman-Williams:KnottedPeriodicOrbitsInDynamicalSystemsI, Birman-Kofman:NewTwistOnLorenzLinks}), which is intuitively speaking, the number of trips that the link travels in the left and right ears of the Lorenz template (\reffig{LorenzTemplate}, left).  The trip number can be easily retrieved from the code word of the Lorenz link embedded in the Lorenz template. See \reffig{LorenzTemplate}, right.

\begin{figure}
% left:
%% Creator: Inkscape 1.1.2 (b8e25be833, 2022-02-05), www.inkscape.org
%% Accompanies image file 'LorenzTemplate.pdf' (pdf, eps, ps)
\begingroup%
  \makeatletter%
  \providecommand\color[2][]{%
    \errmessage{(Inkscape) Color is used for the text in Inkscape, but the package 'color.sty' is not loaded}%
    \renewcommand\color[2][]{}%
  }%
  \providecommand\transparent[1]{%
    \errmessage{(Inkscape) Transparency is used (non-zero) for the text in Inkscape, but the package 'transparent.sty' is not loaded}%
    \renewcommand\transparent[1]{}%
  }%
  \providecommand\rotatebox[2]{#2}%
  \newcommand*\fsize{\dimexpr\f@size pt\relax}%
  \newcommand*\lineheight[1]{\fontsize{\fsize}{#1\fsize}\selectfont}%
  \ifx\svgwidth\undefined%
    \setlength{\unitlength}{141.44883994bp}%
    \ifx\svgscale\undefined%
      \relax%
    \else%
      \setlength{\unitlength}{\unitlength * \real{\svgscale}}%
    \fi%
  \else%
    \setlength{\unitlength}{\svgwidth}%
  \fi%
  \global\let\svgwidth\undefined%
  \global\let\svgscale\undefined%
  \makeatother%
  \begin{picture}(1,0.46887706)%
    \lineheight{1}%
    \setlength\tabcolsep{0pt}%
    \put(0,0){\includegraphics[width=\unitlength,page=1]{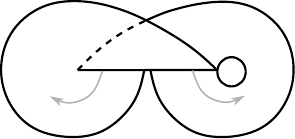}}%
    \put(0.24443091,0.07011143){\color[rgb]{0.6,0.6,0.6}\makebox(0,0)[lt]{\lineheight{1.25}\smash{\begin{tabular}[t]{l}$x$\end{tabular}}}}%
    \put(0.68548454,0.07368241){\color[rgb]{0.6,0.6,0.6}\makebox(0,0)[lt]{\lineheight{1.25}\smash{\begin{tabular}[t]{l}$y$\end{tabular}}}}%
    \put(0,0){\includegraphics[width=\unitlength,page=2]{LorenzTemplate.pdf}}%
  \end{picture}%
\endgroup%

\hspace{16mm} 
% right: 
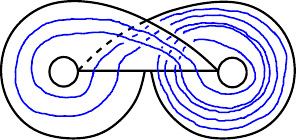

\caption{Left: Lorenz template. Right: The Lorenz link represented by the code word $x^1 y^2 x^1 y^3$, which has trip number two. Note that $x^1 y^2$ can be seen as one ``trip'' and $x^1 y^3$ can be seen as another ``trip''.
\label{Fig:LorenzTemplate}}
\end{figure}

Apart from showing that Lorenz links are precisely T-links in~\cite{Birman-Kofman:NewTwistOnLorenzLinks}, Birman and Kofman also provided a formula for computing braid index of T-links from the $T$-link parameters in~\cite[Corollary~8]{Birman-Kofman:NewTwistOnLorenzLinks}. 

Since the braid index of Lorenz links and T-links are easily computable, this is another reason why we defined the associated Lorenz links of generalised T-links to provide upper volume bounds for all link complements in $S^3$. For each admissible choice of \(n\), the link \(L_n^+\) is a classical Lorenz link (or equivalently, T-link), and therefore
\(\operatorname{br}(L_n^+)\) is a computable quantity determined by the
associated positive T-link parameters. We define the \emph{Lorenz braid index} of
\(L\), denoted by \(\operatorname{lbi}(L)\), to be the minimum of these
braid indices over all generalised T-link representations of \(L\) and
all admissible choices of \(n\). Thus \(\operatorname{lbi}(L)\) is an invariant of the link \(L\), rather
than a complexity attached to a particular representative. Our main result
shows that there is an upper bound for the sum of volumes of all hyperbolic pieces of $S^3\setminus L$ that is quadratic in terms of the Lorenz braid index: 

\begin{named}{\refthm{VolumeLorenzBraidIndex}} 
Let \(L\subset S^3\) be a link, and let \(\operatorname{Vol}(L)\) denote the
sum of the volumes of the hyperbolic pieces in the JSJ decomposition of
\(S^3\setminus L\). Then
\[
\operatorname{Vol}(L)
\leq
12v_{\mathrm{tet}}
\bigl(
\operatorname{lbi}(L)^2
+
3\operatorname{lbi}(L)
\bigr),
\]
where \(v_{\mathrm{tet}}\) denotes the volume of the regular ideal
tetrahedron.
\end{named} 

The Lorenz braid index is a braid-theoretic complexity different from the classical braid index. The classical braid index measures the minimum number of strands needed to realise a link as a closed braid.
The Lorenz braid index instead measures the braid index of the associated
positive Lorenz link obtained from a generalised T-link representation,
after removing the relevant full twists.  

This distinction is geometrically meaningful. Large amounts of twisting in a fixed region can often be realised by Dehn filling on a fixed parent link, and Dehn filling does not increase simplicial volume. Thus, such twisting may increase classical braid-theoretic complexity without contributing to volume growth. The Lorenz braid index is designed to measure the remaining Lorenz complexity after this twisting has been separated off, and our main theorem shows that this quantity is sufficient to bound the total hyperbolic volume from above.

To show that this distinction is significant, we construct an explicit family of hyperbolic Lorenz knots for which the classical braid index and the Seifert genus both tend to infinity, while the Lorenz braid index remains uniformly bounded. By the volume bound above, the volumes of these knot complements are therefore uniformly bounded.

\begin{named}{Theorem~\ref{thm:bounded-volume-unbounded-braid}} 
There exists a sequence of hyperbolic Lorenz knots \(K_j\subset S^3\) such
that
\[
\operatorname{br}(K_j)\to\infty
\qquad\text{and}\qquad
g(K_j)\to\infty,
\]
while the Lorenz braid indices \(\operatorname{lbi}(K_j)\) are uniformly
bounded. Consequently, the volumes of \(S^3\setminus K_j\) are uniformly
bounded.
\end{named} 

This family gives a concrete Lorenz-knot realisation of the principle
that classical topological complexity may grow without forcing volume
growth, in the spirit of the independence between volume and genus-\(g\)
bridge numbers established by Purcell and Zupan~\cite{PurcellZupan}. In our
examples, the classical braid index and the Seifert genus tend to infinity,
whereas the associated Lorenz links have uniformly bounded braid index.
Thus, in this case, the Lorenz braid index detects the bounded geometric complexity
responsible for the volume bound, while the classical braid index and genus
do not.  

The proof of the main volume bound has two main ingredients. The first is a
generalisation of the bunch algorithm~\cite{Hui-Rodriguez:BunchesUpperVolume} of the last two authors from Lorenz templates to Lorenz-like templates. This construction associates to any link embedded in such a template a parent manifold whose simplicial volume is bounded above quadratically in terms of the trip number. Since the original link complement is obtained from this parent manifold by Dehn filling, and Dehn filling does not increase simplicial volume, this gives a quadratic upper bound for the sum of the volumes of the hyperbolic pieces of the original link complement.

The second ingredient is the passage from generalised T-links to their
associated Lorenz links. If the generalised T-link has negative twist(s) in the last block, we can obtain an associated Lorenz link by adding some positive full twists. While the braid index of the generalised T-links may be hard to compute, any of its associated Lorenz link has braid index that is easily computable. Minimising this braid index over all generalised T-link representations
and all admissible choices of full twists gives the Lorenz braid index. By combining the two ingredients, we obtain the stated quadratic volume bound.

The paper is organised as follows. In
Section~\ref{Sec:Bunch&Templates}, we generalised the bunch algorithm in~\cite{Hui-Rodriguez:BunchesUpperVolume} for constructing links embedded in Lorenz-like templates. In Section~\ref{Sec:VolBdForAllLinks}, we use the generalised bunch algorithm to construct parent manifolds for links embedded in Lorenz-like templates and prove a volume bound in terms of trip number. In Section~\ref{Sec:GeneralisedTLinks}, we relate this template-theoretic quantity to generalised T-links: we define associated Lorenz links, introduce the Lorenz braid index, and prove that it agrees with the minimal trip number. This gives the main volume bound in terms of the Lorenz braid index. Finally, in Section~\ref{Sec:BoundedVolumeUnboundedBraid}, we construct a sequence of hyperbolic Lorenz knots with unbounded classical braid index and Seifert genus, but bounded Lorenz braid index and volume.

\subsection*{Acknowledgements}
We thank Michel Boileau for helpful conversations.

\section{The generalised bunch algorithm and Lorenz-like templates} \label{Sec:Bunch&Templates}  

A \emph{template} is a compact branched surface with nonempty boundary.  Such a two-dimensional object is associated with a semiflow constructed locally from the joining and splitting charts as shown in \cite[Figure~2.4]{Ghrist-Holmes-Sullivan:KnotsALinksInThreeDimFlows}.  A fundamental example of templates is the \emph{Lorenz template}, which is a geometric model introduced by Guckenheimer and Williams for studying Lorenz flow~\cite{Guckenheimer-Williams:StructuralStabilityOfLorenzAttractors,Williams:StructureOfLorenzAttractors}.  Tucker justified this model later in~\cite{Tucker:Smale14thProblem}. Lorenz links can thus be viewed as links embedded in the Lorenz template, which can be studied using symbolic dynamics (see \cite{Williams:StructureOfLorenzAttractors, Birman-Williams:KnottedPeriodicOrbitsInDynamicalSystemsI}). 

%In \cite{Hui-Rodriguez:BunchesUpperVolume}, the authors developed the bunch algorithm for constructing  Lorenz links in the Lorenz template when a code word is given. 

In \cite{Hui-Rodriguez:BunchesUpperVolume}, the bunch algorithm was developed to construct Lorenz links in the Lorenz template when a code word is provided. The bunch algorithm provided the an efficient way of building a parent manifold and finding an upper volume bound for Lorenz link complements. Such an upper bound is independent of the word exponents and is quadratic in the trip number, which is the same as the braid index of the Lorenz link \cite{Franks-Williams:BraidsAndJonesPolynomial, Birman-Williams:KnottedPeriodicOrbitsInDynamicalSystemsI}. It turns out that the bunch algorithm can be generalised to infinitely many \emph{universal} templates, which are templates embedded in the $3$-sphere that contain a representative of each ambient isotopy class of links in the $3$-sphere (see also \cite[Definition 3.2.15]{Ghrist-Holmes-Sullivan:KnotsALinksInThreeDimFlows}).  We will explain more in the following. 

\begin{definition}\label{Def:LOrenzLikeTemplate}
    Let $u$, $v$ be integers. The \emph{$(u,v)$-Lorenz-like template $\mathscr{L}(u,v)$} is a template obtained by cutting along the branch line of the Lorenz template, and then gluing the left strip back with \(u\) half twists, and gluing the right strip back with \(v\) half twists.  \reffig{LorenzLikeTemplate} shows an example of Lorenz-like template with positive half twists. Negative half twists can be obtained by switching the crossings. 

    A link embedded in $\mathscr{L}(u,v)$ is called a \emph{$(u,v)$-Lorenz-like link}, or simply \emph{Lorenz-like link} when the context is clear.  
\end{definition} 

\begin{figure}
    \centering
    %% Creator: Inkscape 1.1.2 (b8e25be833, 2022-02-05), www.inkscape.org
%% Accompanies image file 'LorenzLikeTemplate.pdf' (pdf, eps, ps)
\begingroup%
  \makeatletter%
  \providecommand\color[2][]{%
    \errmessage{(Inkscape) Color is used for the text in Inkscape, but the package 'color.sty' is not loaded}%
    \renewcommand\color[2][]{}%
  }%
  \providecommand\transparent[1]{%
    \errmessage{(Inkscape) Transparency is used (non-zero) for the text in Inkscape, but the package 'transparent.sty' is not loaded}%
    \renewcommand\transparent[1]{}%
  }%
  \providecommand\rotatebox[2]{#2}%
  \newcommand*\fsize{\dimexpr\f@size pt\relax}%
  \newcommand*\lineheight[1]{\fontsize{\fsize}{#1\fsize}\selectfont}%
  \ifx\svgwidth\undefined%
    \setlength{\unitlength}{144.61663982bp}%
    \ifx\svgscale\undefined%
      \relax%
    \else%
      \setlength{\unitlength}{\unitlength * \real{\svgscale}}%
    \fi%
  \else%
    \setlength{\unitlength}{\svgwidth}%
  \fi%
  \global\let\svgwidth\undefined%
  \global\let\svgscale\undefined%
  \makeatother%
  \begin{picture}(1,0.7315248)%
    \lineheight{1}%
    \setlength\tabcolsep{0pt}%
    \put(0,0){\includegraphics[width=\unitlength,page=1]{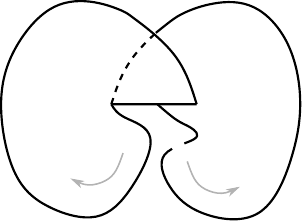}}%
    \put(0.3075093,0.06912099){\color[rgb]{0.6,0.6,0.6}\makebox(0,0)[lt]{\lineheight{1.25}\smash{\begin{tabular}[t]{l}$x$\end{tabular}}}}%
    \put(0.67403551,0.04283614){\color[rgb]{0.6,0.6,0.6}\makebox(0,0)[lt]{\lineheight{1.25}\smash{\begin{tabular}[t]{l}$y$\end{tabular}}}}%
    \put(0,0){\includegraphics[width=\unitlength,page=2]{LorenzLikeTemplate.pdf}}%
  \end{picture}%
\endgroup%
    \caption{Lorenz-like template $\mathscr{L}(1,2)$}
    \label{Fig:LorenzLikeTemplate}
\end{figure} 

\begin{proposition}[\cite{Ghrist:BranchedTwoMfdsSupportingAllLinks} and Proposition~3.2.16 in \cite{Ghrist-Holmes-Sullivan:KnotsALinksInThreeDimFlows}]\label{Prop:UniversalTemplates}
If $v$ is a negative integer, then the Lorenz-like template $\mathscr{L}(0,v)$ is universal. 
\end{proposition} 

Take $\mathscr{L}(0,-2)$ as an example. \refprop{UniversalTemplates} implies that every link in the $3$-sphere can be ambient-isotoped to a link embedded in $\mathscr{L}(0,-2)$.

Just like code words of modular links (see \cite{Birman-Williams:KnottedPeriodicOrbitsInDynamicalSystemsI} and \cite[Definition~3.1~(4)]{Hui-Rodriguez:BunchesUpperVolume} for more details), which are equivalent to Lorenz links, we have code words and labelled code words for Lorenz-like links using the symbols $x$ and $y$ as shown in \reffig{LorenzLikeTemplate}. 

Starting form here, we will use terminologies and notations similar to those in \cite{Hui-Rodriguez:BunchesUpperVolume} to state and prove the generalised bunch algorithm.  

Let $L$ be a Lorenz-like link with $c$ link components. 
We denote the corresponding $c$ labelled code words for $L$ by the following: 

\begin{align*}  \label{Words} 
\begin{dcases}
    w_1 &= x_{1}^{k_{1}} y_{1}^{l_{1}} x_{2}^{k_{2}} y_{2}^{l_{2}} \ldots x_{n_1}^{k_{n_1}} y_{n_1}^{l_{n_1}}, \\ 
w_2 &= x_{n_1+1}^{k_{n_1+1}} y_{n_1+1}^{l_{n_1+1}} x_{n_1+2}^{k_{n_1+2}} y_{n_1+2}^{l_{n_1+2}} \ldots x_{n_1+n_2}^{k_{n_1+n_2}} y_{n_1+n_2}^{l_{n_1+n_2}},  \\ 
&\ldots \\ 
w_c &= x_{\Sigma(c) + 1}^{k_{\Sigma(c) + 1}} y_{\Sigma(c) + 1}^{l_{\Sigma(c) + 1}} x_{\Sigma(c) + 2}^{k_{\Sigma(c) + 2}} y_{\Sigma(c) + 2}^{l_{\Sigma(c) + 2}} \ldots x_{\overline{n}}^{k_{\overline{n}}} y_{\overline{n}}^{l_{\overline{n}}}
\end{dcases} \tag{*}
\end{align*} 
The notation $\Sigma(c)$ represents the sum $\Sigma_{i=1}^{c-1} n_i$ of the word periods (i.e. trip numbers of the link components) for words $w_1, \ldots, w_{c-1}$ if $c>1$.  The symbol $\overline{n}$ denotes the sum of all word periods for $L$, that is, $\overline{n} = \Sigma(c+1) = \Sigma_{i=1}^{c} n_i$. Note that $\overline{n}$ is the trip number of the link $L$ in a Lorenz-like template $\mathscr L(u,v)$. For clarity purposes, from now on we will denote $\overline{n}$ by 
$$\tau_{\mathscr L(u,v)}(L)$$ unless otherwise specified. By abuse of notation, the $L$ in the bracket means a certain template representative of $L$ instead of the ambient isotopy class of $L$. 

Similar to the explanation in the caption of Figure~\ref{Fig:LorenzTemplate} (right), we can view each trip as the arc corresponding to each subword $x_i^{k_i} y_i^{l_i}$.

\begin{theorem}[Bunches in Lorenz-like links]\label{Thm:UnionBunchesOfTurns_Gen}
    Let $u$, $v$ be even integers. 
    Each $(u,v)$-Lorenz-like link $L$ with labelled code words in (\ref{Words}) is ambient isotopic to a union of bunches of turns that are ordered from left to right in the split template according to the order of turns listed in each bunch, and the union of bunches satisfy both conditions below:
\begin{enumerate}
    \item For any $i\in [1,\overline{n}]\cap\ZZ$, the $x_i^{k_i}$-arc starts from a point in the $-{k_i}^{\textup{th}}$ interval. 
    \item For any $j\in [1,\overline{n}]\cap\ZZ$, the $y_j^{l_j}$-arc starts from a point in the $+{l_j}^{\textup{th}}$ interval. 
\end{enumerate}
\end{theorem} 

\begin{proof}
    Since $u$, $v$ are even integers, $\frac{u}{2}$ and $\frac{v}{2}$ are integers. There are $\frac{u}{2}$ full twists in the left strip of the $(u,v)$-Lorenz-like template and $\frac{v}{2}$ full twists in the right strip. 
    As the number of full twists are integer-valued, we can get a split template that is the same as that of the Lorenz template (see example in \cite[Figure~2, bottom]{Hui-Rodriguez:BunchesUpperVolume}), and the corresponding sets of code words are the same for both templates. The statement then follows from the proof of \cite[Theorem~3.7]{Hui-Rodriguez:BunchesUpperVolume}. 
\end{proof} 

Since the full twists of strips do not affect the order of bases within each bunch, the concept \emph{order within full bunches} in Lorenz link introduced in \cite{Hui-Rodriguez:BunchesUpperVolume} can thus apply to Lorenz-like links. Readers may refer to \cite[Section~3.2]{Hui-Rodriguez:BunchesUpperVolume} for more details and examples of finding order within full bunches given the code words of a Lorenz-like link. 

Following from \refthm{UnionBunchesOfTurns_Gen} and \cite[Proposition~3.8]{Hui-Rodriguez:BunchesUpperVolume}, we have the following generalised bunch algorithm for constructing the Lorenz-like link in $\mathscr{L}(u,v)$ (where $u$, $v$ are even integers) whenever code words are given:

\begin{corollary}[The generalised bunch algorithm] \label{Cor:Algorithm} 
Let $u$, $v$ be even integers. 
Let $L$ be a $(u,v)$-Lorenz-like link embedded in $\mathscr{L}(u,v)$ with the sum of all word periods denoted by $\overline{n}$. Suppose $K$ is a link component of $L$ with labelled code word $w = x_1^{k_1} y_1^{l_1} \ldots x_n^{k_n} y_n^{l_n}$. The link component $K$ can be constructed in the split template using the following steps: 
\begin{enumerate}
    \item Mark the branch line with integers $-k_{\mu}, \ldots, +l_{\lambda}$, where $k_{\mu}$ and $l_{\lambda}$ are a maximal $x$- and $y$-exponents among all words associated to $L$ respectively. 
    
    \item Determine the orders within the full bunches for the Lorenz-like link \(L\) using \cite[Proposition~3.8]{Hui-Rodriguez:BunchesUpperVolume}. In each negative (or positive  resp.) interval, draw $\overline{n}$ points with $x$-base (or $y$-base resp.) labels from left to right according to the order within the full $x$-bunch (or full $y$-bunch resp.) of the link (NOT the link component). 
    
    \item Construct $K$ by first drawing the $x_1^{k_1}$-arc, which starts from the $x_1$-point in the $-{k_1}^{\textup{th}}$ interval and ends at the $y_1$-point in the $+{l_1}^{\textup{th}}$ interval. Then, draw the $y_1^{l_1}$-arc which starts from the $y_1$-point in the $+{l_1}^{\textup{th}}$ interval and ends at the $x_2$-point in the $-{k_2}^{\textup{th}}$ interval. Continue drawing the arcs for subwords $x_2^{k_2}, y_2^{l_2}, \ldots, x_n^{k_n}, y_n^{l_n}$ with known starting and ending points to obtain $K$ in the split template. 
\end{enumerate}
    
By repeating Step (3) for each link component of $L$ and making $u$ and $v$ half twists on the left and right strips of the split template respectively, the Lorenz-like link $L$ embedded in the Lorenz-like template can be obtained. 
\end{corollary}

\section{An upper bound for the volumes of all link complements in the $3$-sphere} \label{Sec:VolBdForAllLinks} 

Using the techniques developed in \refsec{Bunch&Templates} and \cite{Hui-Rodriguez:BunchesUpperVolume}, we give an upper bound for the volume sums of all hyperbolic pieces in any link complement in the $3$-sphere. Such an upper bound is quadratic in the minimal trip number (see~\refcor{UpperBdTripNumber}). 

One of the results we use for finding the upper volume bound is Theorem~1.5 in \cite{Cremaschi-RodriguezMigueles:HypOfLinkCpmInSFSpaces}, which can be stated in the following.

\begin{theorem}[Theorem~1.5 in \cite{Cremaschi-RodriguezMigueles:HypOfLinkCpmInSFSpaces}]\footnote{Abuse of notation: The symbol $\calL$ in \refthm{SelfIntersectionNumber} is considered as an embedding in the function composition and as a set otherwise. }  \label{Thm:SelfIntersectionNumber}
Let $M$ be a Seifert-fibred space over a hyperbolic $2$-orbifold $B$ with $\calP\from M\to B$ defined to be the corresponding projection map. If $\calL$ is a link embedded in $M$ such that $\calP\circ \calL$ is a collection of loops that intersect themselves only transversely and finitely many times with exactly two pre-image points for each self-intersection point, 
% If $B\setminus\calP(L)$ is a union of finitely many discs with finitely many punctures, 
then the simplicial volume of $M\setminus \calL$ satisfies the following inequality: 
\[\norm{M\setminus \calL} \leq 8\ \iota(\calP(\calL),\calP(\calL)), \] 
where $\iota(\calP(\calL),\calP(\calL))$ is the number of self-intersections of $\calP(\calL)$. 
\end{theorem} 

For details on simplicial volume and its relationship with hyperbolic volume, readers may refer to \cite[6.5.2]{Thurston:Geom&TopOf3Mfd}, \cite[6.5.4]{Thurston:Geom&TopOf3Mfd} and \cite[6.1.7]{Thurston:Geom&TopOf3Mfd}. A fact that we will use is: if $M$ is an oriented, finite-volume hyperbolic manifold, then $v_{\text{tet}}\norm{M} = \mathrm{Vol}(M)$, where $v_{\textup{tet}}\approx 1.01494$ denotes the volume of the regular ideal tetrahedron and $\norm{M}$ is the simplicial volume of $M$.

\begin{definition}[Parent manifold of each Lorenz-like link] \label{Def:ParentMfd}
    The \emph{parent manifold $S^3\setminus(\calL_{\textup{SF}}\cup \calL_{L})$ of the Lorenz-like link $L$} is the complement of the link $\calL_L$ in the Seifert-fibred space $S^3\setminus \calL_{\textup{SF}}$ (see \cite[Definition~4.3]{Hui-Rodriguez:BunchesUpperVolume}), where $\calL_L$ consists of the following: 
    \begin{itemize}
        \item the parent link $L_{\mathrm{p}}$ (see \cite[Definition~4.5]{Hui-Rodriguez:BunchesUpperVolume} for its construction), 
        \item the unknots $V_{a_1}, \ldots, V_{a_{\#l}}$, where each $V_j$ encircles all the blue vertical line segment(s) (as defined in Step~(3) of \cite[Definition~4.5]{Hui-Rodriguez:BunchesUpperVolume}) that is/are going to intersect the annulus $A^y_j$, and
        \item the unknots $V_L$ and $V_R$ that encircle the left and right sets of annuli respectively.  
    \end{itemize} 
\end{definition} 

See \reffig{LinkCompleInSFSpaceResized_TwoMoreUnknots} for example. 

\begin{figure}
%% Creator: Inkscape 1.1.2 (b8e25be833, 2022-02-05), www.inkscape.org
%% Accompanies image file 'LinkCompleInSFSpaceResized_TwoMoreUnknots.pdf' (pdf, eps, ps)
\begingroup%
  \makeatletter%
  \providecommand\color[2][]{%
    \errmessage{(Inkscape) Color is used for the text in Inkscape, but the package 'color.sty' is not loaded}%
    \renewcommand\color[2][]{}%
  }%
  \providecommand\transparent[1]{%
    \errmessage{(Inkscape) Transparency is used (non-zero) for the text in Inkscape, but the package 'transparent.sty' is not loaded}%
    \renewcommand\transparent[1]{}%
  }%
  \providecommand\rotatebox[2]{#2}%
  \newcommand*\fsize{\dimexpr\f@size pt\relax}%
  \newcommand*\lineheight[1]{\fontsize{\fsize}{#1\fsize}\selectfont}%
  \ifx\svgwidth\undefined%
    \setlength{\unitlength}{374.7297126bp}%
    \ifx\svgscale\undefined%
      \relax%
    \else%
      \setlength{\unitlength}{\unitlength * \real{\svgscale}}%
    \fi%
  \else%
    \setlength{\unitlength}{\svgwidth}%
  \fi%
  \global\let\svgwidth\undefined%
  \global\let\svgscale\undefined%
  \makeatother%
  \begin{picture}(1,0.42487583)%
    \lineheight{1}%
    \setlength\tabcolsep{0pt}%
    \put(0,0){\includegraphics[width=\unitlength,page=1]{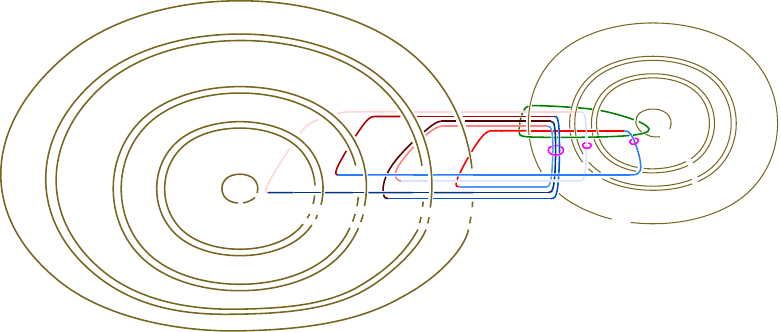}}%
    \put(0.71957315,0.21132254){\color[rgb]{1,0,1}\makebox(0,0)[lt]{\lineheight{1.25}\smash{\begin{tabular}[t]{l}\fontsize{5pt}{1em}$V_2$\end{tabular}}}}%
    \put(0.75310614,0.21933034){\color[rgb]{1,0,1}\makebox(0,0)[lt]{\lineheight{1.25}\smash{\begin{tabular}[t]{l}\fontsize{5pt}{1em}$V_3$\end{tabular}}}}%
    \put(0.81830144,0.22733654){\color[rgb]{1,0,1}\makebox(0,0)[lt]{\lineheight{1.25}\smash{\begin{tabular}[t]{l}\fontsize{5pt}{1em}$V_6$\end{tabular}}}}%
    \put(0,0){\includegraphics[width=\unitlength,page=2]{LinkCompleInSFSpaceResized_TwoMoreUnknots.pdf}}%
    \put(0.67231021,0.1001778){\color[rgb]{0.50588235,0.50588235,0.50588235}\makebox(0,0)[lt]{\lineheight{1.25}\smash{\begin{tabular}[t]{l}$U$\end{tabular}}}}%
  \end{picture}%
\endgroup%

\caption{Parent manifold $S^3\setminus (\calL_{\textup{SF}}\cup \calL_K)$ of the Lorenz-like knot $K$ with code word $x^{10}y^2x^5y^2x^7y^6x^2y^2x^5y^3$, where $\calL_{\textup{SF}}$ consists of the unknot $U$ and all annuli boundaries, and $\calL_K$ consists of the parent knot $K_{\mathrm{p}}$ (with colours red and blue), the unknots $V_L$ and $V_R$ (green), and the unknots $V_2$, $V_3$, $V_6$ (magenta).   The Seifert-fibred space is $S^3\setminus \calL_{\textup{SF}}$ with base space equal to a $14$-punctured disc bounded by $U$. 
\label{Fig:LinkCompleInSFSpaceResized_TwoMoreUnknots}}
\end{figure}

\begin{proposition} \label{Prop:DehnFillParentMfd_Gen}
    Let $u$, $v$ be even integers. Each $(u,v)$-Lorenz-like link complement $S^3\setminus L$ can be obtained by applying Dehn fillings on its parent manifold $S^3\setminus(\calL_{\textup{SF}}\cup \calL_L)$.  
\end{proposition}  

\begin{proof}
Since the generalised bunch algorithm (\refcor{Algorithm}) holds, part of the argument follows similarly as the argument in \cite[Proposition~4.8]{Hui-Rodriguez:BunchesUpperVolume}. The only differences between the parent manifold in \cite[Proposition~4.8]{Hui-Rodriguez:BunchesUpperVolume} and the parent manifold here are:
\begin{enumerate}
    \item The former leaves the trefoil knot component not Dehn-filled while the later does not drill out the trefoil knot component. 
    \item The former does not have twists on the two strips of the Lorenz template while the later may have full twists, which require extra Dehn fillings on $V_L$ and $V_R$. 
\end{enumerate}

Apart from the Dehn fillings mentioned in the argument of \cite[Proposition~4.8]{Hui-Rodriguez:BunchesUpperVolume}, we apply $(\frac{1}{u/2})$-Dehn filling on $V_L$ and apply $(\frac{1}{v/2})$-Dehn filling on $V_R$. The desired complement of the Lorenz-like link $L$ can thus be obtained.  
\end{proof}

\begin{theorem} \label{Thm:UpperBdTripNumber}
Let $u$, $v$ be even integers.  
Let $L$ be a link in the $3$-sphere embedded in a Lorenz-like template $\mathscr{L}(u, v)$ with trip number equal to $\overline{n}$. Let $\mathrm{Vol}(L)$ denote the sum of the volumes of the hyperbolic pieces of $S^3\setminus L$. Then, 
\[\mathrm{Vol}(L) \leq 12\ v_{\textup{tet}}(\overline{n}^2+3\overline{n}), \]
where $v_{\textup{tet}}\approx 1.01494$ is the volume of the regular ideal tetrahedron. 
\end{theorem}  

\begin{proof} 
By \refprop{DehnFillParentMfd_Gen}, the complement of the $(u,v)$-Lorenz-like link $L$ in the $3$-sphere can be obtained by applying a finite number of Dehn fillings on its parent manifold $S^3\setminus(\calL_{\textup{SF}}\cup \calL_L)$ (see \refdef{ParentMfd}).  Since Dehn fillings do not increase simplicial volume~\cite[Proposition~6.5.2]{Thurston:Geom&TopOf3Mfd}, it suffices to find an upper bound for the simplicial volume of the parent manifold. 

Observe that the parent link $L_{\mathrm{p}}$ of the $(u,v)$-Lorenz-like link $L$ is the same as the parent link of the Lorenz link obtained by removing the $u$~half twists (left strip) and $v$~half twists (right strip) of $L$. By \cite[Lemma 4.10]{Hui-Rodriguez:BunchesUpperVolume}, the parent link $L_{\mathrm{p}}$ can be continuously deformed to a location such that the number of self-intersections of its projection~$\calP(L_{\mathrm{p}})$ is at most $\frac{3}{2}\overline{n}(\overline{n}-1)$. 

The unknot(s) $V_{a_1}, \ldots, V_{a_{\#l}}$ for untwisting undesired full twist(s) contribute at most $2\overline{n}$ self-crossings of $\calP(\calL_L)$ because there are a total of $\overline{n}$ (blue) vertical strand(s) approaching the $y$-annuli and the number of crossings that each unknot creates is twice the number of strands it encircles.  
(See \reffig{LinkCompleInSFSpaceResizedProjection_TwoMore} for example, where the parent link $L_{\mathrm{p}}$ is the parent knot $K_{\mathrm{p}}$ and $V_{a_1}, \ldots, V_{a_{\#l}}$ is $V_2, V_3, V_6$ in that case.)

Hence, the total number of self-crossings in $\calP(\calL_L\setminus (V_L\cup V_R))$ is 
\[\frac{3}{2}\overline{n}(\overline{n}-1) + 2\overline{n}= \frac{3}{2}\overline{n}^2+\frac{1}{2}\overline{n}.\]  

With the addition of the two loops $V_L$ and $V_R$ in the parent manifold (See \reffig{LinkCompleInSFSpaceResizedProjection_TwoMore} for example), the total number of self-crossings in $\calP(\calL_L)$ becomes 
\[\frac{3}{2}\overline{n}^2+\frac{1}{2}\overline{n} +  2*2*\overline{n} = \frac{3}{2}\overline{n}^2+\frac{9}{2}\overline{n}\]

By \cite[Theorem~1.5]{Cremaschi-RodriguezMigueles:HypOfLinkCpmInSFSpaces}, we have 
\[\norm{S^3\setminus(\calL_{\textup{SF}}\cup \calL_L)} \leq 8 \ \iota(\calP(\calL_L),\calP(\calL_L)) \leq 12(\overline{n}^2+3\overline{n}).\]

Since Dehn fillings do not increase volumes \cite[Theorem~6.5.6]{Thurston:Geom&TopOf3Mfd}, it follows together with \cite[Theorem~6.5.5]{Thurston:Geom&TopOf3Mfd} that 
\[\mathrm{Vol}(L) = \mathrm{Vol}((S^3 \setminus L)^{\text{hyp}}) \leq v_{\textup{tet}} \norm{S^3\setminus(\calL_{\textup{SF}}\cup \calL_L)}  \leq  12\ v_{\textup{tet}}(\overline{n}^2+3\overline{n}),\] 
an upper bound that is quadratic in terms of the trip number. 
\end{proof} 

\begin{figure}
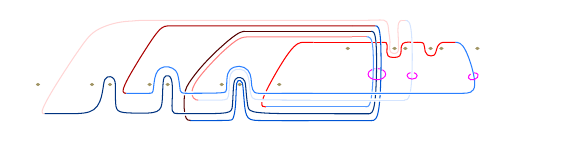
\caption{The projection of the link $\calL_K = K_{\mathrm{p}} \cup V_L \cup V_R \cup V_2 \cup V_3\cup V_6$ in \reffig{LinkCompleInSFSpaceResized_TwoMoreUnknots} in the $14$-punctured disc bounded by the unknot $U$.  The green loops $V_L$ and $V_R$ produce $2*2*5=20$ more crossings in the link projection. 
\label{Fig:LinkCompleInSFSpaceResizedProjection_TwoMore}}
\end{figure} 

Note that a link can be embedded in infinitely many universal templates and the trip number of the same link may differ in different templates (see Proposition~\ref{Prop:TripNumberDependsOnRepresentative} for a related result). 

\refprop{UniversalTemplates} states that the Lorenz-like template $\mathscr{L}(0,v)$ is universal for any negative integer $v$, and the generalised bunch algorithm works for all these universal templates with $v$ being an even integer.  Motivated from the above observations, we define the following invariant of link: 

\begin{definition}\label{Def:MinTripNumber}
For any link \(L\subset S^3\), the \emph{minimal trip number} of \(L\),
denoted by \(\widetilde n(L)\), is
\[
\widetilde n(L)
=
\min
\tau_{\mathscr L(0,2m)}(L'),
\]
where the minimum is taken over all negative integers \(m\) and all representatives
\(L'\) of \(L\) embedded in the Lorenz-like template \(\mathscr L(0,2m)\).
\end{definition}

The following corollary follows from \refthm{UpperBdTripNumber}. 

\begin{corollary} \label{Cor:UpperBdTripNumber}
If $L$ is a link in the $3$-sphere  with minimal trip number $\tilde{n}$, then the following inequality holds: 
\[\mathrm{Vol}(L) \leq 12\ v_{\textup{tet}}(\tilde{n}^2+3\tilde{n}), \]
where $v_{\textup{tet}}\approx 1.01494$ is the volume of the regular ideal tetrahedron.  
\end{corollary}  

\begin{proof} 
Let $m$ be any negative integer. Let $L'$ be any template representative of $L$ in the Lorenz-like template $\mathscr{L}(0,2m)$ with trip number $\overline{n}$. By Theorem~\ref{Thm:UpperBdTripNumber},
$$\mathrm{Vol}(L) \leq 12\ v_{\textup{tet}}(\overline{n}^2+3\overline{n}).$$
Since the inequality holds for any $\overline{n}$ for any template representative $L'$ in any Lorenz-like template of the form $\mathscr{L}(0,2m)$, we have
$$\mathrm{Vol}(L) \leq 12\ v_{\textup{tet}}(\tilde{n}^2+3\tilde{n}),$$
where $\tilde{n}$ is the minimal trip number of $L$.
\end{proof}

\section{Generalised T-links, associated Lorenz links, and volume}
\label{Sec:GeneralisedTLinks}

The goals of this section are to connect the universal braid description of
links by generalised T-links with the volume estimates coming from
Lorenz-like templates, and to show a main result (Theorem~\ref{Thm:VolumeLorenzBraidIndex}) in this paper. 

In~\cite[Theorem~3.10]{dePaiva-Hui-Rodriguez:GeneralisedTLinks}, we know that some generalised T-links coincide with Lorenz-like links in universal Lorenz-like template of the form~$\mathscr{L}(0,2m)$ with $m$ being a negative integer. The volume bound in the last section can thus be applied to those generalised T-links. Since the trip number of a Lorenz link is equal to its braid index~\cite{Franks-Williams:BraidsAndJonesPolynomial} while that of a Lorenz-like link may not, this is why an associated Lorenz link comes into play. 

A generalised T-link may contain a final twisting block which contains negative twists. After adding an appropriate number of full twists that cancel the negative twists, we obtain a classical T-link, which is also a Lorenz link~\cite{Birman-Kofman:NewTwistOnLorenzLinks}.
We call this Lorenz link
the \textit{associated Lorenz link}. 
A key observation from \cite{dePaiva-Hui-Rodriguez:GeneralisedTLinks} is that the trip number of the Lorenz-like representative that is ambient isotopic to the generalised T-link is the same as the trip number of the associated Lorenz link. Since the trip number of a Lorenz link agrees with its braid index \cite{Franks-Williams:BraidsAndJonesPolynomial}, the trip number appearing in the previous volume estimate can be replaced by the braid index of the associated
Lorenz link, which is a widely studied link invariant. 

The above observation forms a bridge between the topology of generalised T-links and the geometry of link complements. This session shows that the volume sum of hyperbolic pieces of any link complement in $S^3$ can be bounded above by a quadratic polynomial of the Lorenz braid index, which is a link invariant defined using generalised T-links and their associated Lorenz links. 

Before showing the main result, we will recall some notions and results from~\cite{dePaiva-Hui-Rodriguez:GeneralisedTLinks}. 

\begin{definition}\label{Def:GeneralisedTLink}
Let \(k\) be a positive integer. Suppose $r_1, \ldots, r_k, r_{k+1}$ are integers such that
\[
2\leq r_1<\cdots<r_k<r_{k+1}.
\]
Let \(s_1,\ldots,s_k\) be positive integers, let
\[
s_{k+1}\in\mathbb Z\setminus\{-1,0,1\},
\]
and let \(d\) be a nonegative integer. The \emph{generalised T-link}
\[
T((r_1,s_1),\ldots,(r_k,s_k),(r_{k+1},s_{k+1}),d)
\]
is the closure of the braid on \(d+r_{k+1}\) strands
\[
(\sigma_{d+1}\cdots\sigma_{d+r_1-1})^{s_1}
\cdots
(\sigma_{d+1}\cdots\sigma_{d+r_k-1})^{s_k}
(\sigma_{d+1}\cdots\sigma_{d+r_{k+1}-1})^{s_{k+1}}.
\]
\end{definition}

In Definition~\ref{Def:GeneralisedTLink}, if \(s_{k+1}<0\), then the last factor is 
\[
(\sigma_{d+1}\cdots\sigma_{d+r_{k+1}-1})^{s_{k+1}}
=
(\sigma_{d+r_{k+1}-1}^{-1}\cdots
 \sigma_{d+1}^{-1})^{-s_{k+1}}.
\]

The following theorem is proved in
\cite[Theorem 1.1]{dePaiva-Hui-Rodriguez:GeneralisedTLinks}.

\begin{theorem}\label{Thm:AllLinksGeneralisedTLinks}
Every link in \(S^3\) can be represented as a generalised T-link.
\end{theorem}

We also recall the following family of links, all numerical variables are integers unless otherwise specified.

\begin{definition}\label{Def:TnLink}
Let \(n\in\mathbb Z\), let \(d\geq0\), and let \(R\geq1\). For \(k\geq1\),
suppose that
\[
2\leq r_1<\cdots<r_k\leq R
\]
and that \(s_1,\ldots,s_k>0\). The \emph{\(T^n\)-link}
\[
T^n((r_1,s_1),\ldots,(r_k,s_k),(R;d))
\]
is the closure of the braid on \(d+R\) strands
\[
(\sigma_{1+d}\cdots\sigma_{d+r_1-1})^{s_1}
\cdots
(\sigma_{1+d}\cdots\sigma_{d+r_k-1})^{s_k}
\Delta^n_{R,d+R},
\]
where \(\Delta^n_{R,d+R}\) denotes \(n\) signed half twists on the
\(R\) consecutive strands \(d+1,\ldots,d+R\).

If \(k=0\), the product preceding \(\Delta^n_{R,d+R}\) is omitted, and
the \(T^n\)-link is the closure of \(\Delta^n_{R,d+R}\). In particular,
when \(R=1\), the link \(T^n((1;d))\) is the unlink with \(d+1\)
components.
\end{definition}

We will also use the following correspondence, proved in
\cite[Theorem~3.5]{dePaiva-Hui-Rodriguez:GeneralisedTLinks}.

\begin{theorem}\label{Thm:TnLinksLorenzLikeTemplates}
Let \(n\in\mathbb Z\). A link \(L\subset S^3\) can be embedded in the
Lorenz-like template \(\mathscr L(0,n)\) if and only if \(L\) is a
\(T^n\)-link. Equivalently, the class of links contained in
\(\mathscr L(0,n)\) coincides with the class of \(T^n\)-links.
\end{theorem}

We will also use the following result from
\cite[Theorem 3.9 and Corollary~3.11]{dePaiva-Hui-Rodriguez:GeneralisedTLinks}.

\begin{theorem}\label{Thm:GeneralisedTLinksT2nLinksTemplates}
Let \(n\in\mathbb Z\).

If \(n\leq0\), then the class of \(T^{2n}\)-links coincides with the class
of generalised T-links
\[
T((r_1,s_1),\ldots,(r_k,s_k),
(r_{k+1},s_{k+1}),d)
\]
satisfying
\[
s_{k+1}\geq nr_{k+1}.
\]

If \(n>0\), then the class of \(T^{2n}\)-links coincides with the union of
this class of generalised T-links satisfying $s_{k+1}\geq nr_{k+1}$ and the class of trivial links
represented by \(T((2,1),d)\).

More precisely, if
\[
L=
T((r_1,s_1),\ldots,(r_k,s_k),
(r_{k+1},s_{k+1}),d)
\]
is a generalised T-link satisfying
\[
s_{k+1}\geq nr_{k+1},
\]
then \(L\) is equivalent to the \(T^{2n}\)-link
\[
T^{2n}((r_1,s_1),\ldots,(r_k,s_k),
(r_{k+1},s_{k+1}-nr_{k+1}),(r_{k+1};d)).
\]
If \(s_{k+1}-nr_{k+1}=0\), the block
\[
(r_{k+1},s_{k+1}-nr_{k+1})
\]
is omitted.
\end{theorem}

The Lorenz-like template \(\mathscr L(0,2n)\) is obtained from the Lorenz template \(\mathscr L(0,0)\) by adding \(2n\) signed half twists in the right strip. This gives a natural way to associate a Lorenz link to every link embedded in \(\mathscr L(0,2n)\). 

\begin{definition}\label{Def:AssociatedLorenzLink}
Let \(n\in\mathbb Z\), and let \(L\) be a link embedded in the
Lorenz-like template \(\mathscr L(0,2n)\). The \emph{associated Lorenz
link} of \(L\), denoted by \(L^+\), is the link embedded in the Lorenz
template \(\mathscr L(0,0)\) obtained by removing the \(2n\) signed half
twists from the right strip while keeping the same code word.
\end{definition}

Equivalently, \(L\) is obtained from \(L^+\) by inserting \(2n\) signed
half twists among the strands travelling through the right strip. This
operation changes the twisting in the right strip, but it does not change
the trip number.

We denote by \(\operatorname{br}(L)\) the braid index of a link \(L\),
that is, the minimum number of strands among all closed braid
representatives of \(L\).

\begin{lemma}\label{lem:trip-number-associated-lorenz}
Let \(n\in\mathbb Z\), let \(L\) be a link embedded in the Lorenz-like
template \(\mathscr L(0,2n)\), and let \(L^+\) be its associated Lorenz
link. Then
\[
\tau_{\mathscr L(0,2n)}(L)
=
\tau_{\mathscr L(0,0)}(L^+)
=
\operatorname{br}(L^+).
\]
\end{lemma}

\begin{proof}
Passing from \(L\subset\mathscr L(0,2n)\) to its associated Lorenz link
\(L^+\subset\mathscr L(0,0)\) removes only the \(2n\) signed half twists
from the right strip. This operation preserves the code words
and therefore preserves the word period, which is the same as the trip number. Hence
\[
\tau_{\mathscr L(0,2n)}(L)
=
\tau_{\mathscr L(0,0)}(L^+).
\]

Since \(L^+\) is a Lorenz link, the theorem of Franks and
Williams \cite{Franks-Williams:BraidsAndJonesPolynomial} gives
\[
\tau_{\mathscr L(0,0)}(L^+)
=
\operatorname{br}(L^+).
\]
Combining the two equalities proves the result.
\end{proof}

We adopt the convention that boundary-parallel trivial components are
omitted when passing from a Lorenz-like representative to its associated
Lorenz link.

\begin{corollary}\label{Cor:AssociatedLorenzLinkTn}
Let \(n\in\mathbb Z\), let
\[
2\leq r_1<\cdots<r_k\leq r_{k+1},
\]
let \(s_1,\ldots,s_k>0\), and let \(d\geq0\). Then the T-link
\[
T((r_1,s_1),\ldots,(r_k,s_k))
\]
is the Lorenz link associated with the \(T^n\)-link
\[
T^n((r_1,s_1),\ldots,(r_k,s_k),(r_{k+1};d)).
\]
\end{corollary}

\begin{proof}
The construction in the proof of
 \cite[Theorem~3.5]{dePaiva-Hui-Rodriguez:GeneralisedTLinks} gives a representative of
\[
T^n((r_1,s_1),\ldots,(r_k,s_k),(r_{k+1};d))
\]
in the Lorenz-like template \(\mathscr L(0,n)\), in which the factor
\[
\Delta^n_{r_{k+1},d+r_{k+1}}
\]
is realised by the \(n\) signed half twists in the right strip.

Passing to the associated Lorenz link removes these half twists and,
at the braid level, removes the factor
\(\Delta^n_{r_{k+1},d+r_{k+1}}\). The construction in the proof of \cite[Theorem~3.5]{dePaiva-Hui-Rodriguez:GeneralisedTLinks} shows that, after omitting the boundary-parallel
trivial components, the remaining Lorenz braid is equivalent to
\[
(\sigma_1\cdots\sigma_{r_1-1})^{s_1}
\cdots
(\sigma_1\cdots\sigma_{r_k-1})^{s_k}.
\]
Its closure is
\[
T((r_1,s_1),\ldots,(r_k,s_k)),
\]
which proves the result.
\end{proof}

Note that every link can be embedded in infinitely many universal
Lorenz-like templates. Indeed, Proposition~\ref{Prop:UniversalTemplates}
shows that \(\mathscr L(0,v)\) is universal for every negative integer
\(v\). On the other hand, the trip number depends on the chosen
Lorenz-like representative and not only on the ambient isotopy class of
the link, as the following proposition shows.

\begin{proposition}\label{Prop:TripNumberDependsOnRepresentative}
There exists a link \(L\subset S^3\) admitting Lorenz-like representatives
with arbitrarily large trip number. In particular, the trip number is not
an invariant of the ambient isotopy class of a link.
\end{proposition}

\begin{proof}
Consider the T-link
\[
L=T((2,3)),
\]
which is the closure of the braid \(\sigma_1^3\). Its standard Lorenz template representative has trip number equal to two.

Let \(R\geq3\). Set
\[
A_R=\sigma_1\sigma_2\cdots\sigma_{R-1}.
\] 
The closure of
\[
\sigma_1^2A_R
=
\sigma_1^3\sigma_2\cdots\sigma_{R-1}
\]
is ambient isotopic to \(L\), since successive positive Markov destabilisations on
the rightmost strands reduce this braid to \(\sigma_1^3\).

Fix an integer \(n<0\). We may rewrite the last braid as
\[
\sigma_1^2A_R
=
\sigma_1^2A_R^{\,1-nR}A_R^{\,nR}.
\]
Since
\[
A_R^{\,R}=\Delta_{R,R}^2, 
\] we have
\[
A_R^{\,nR}=\Delta_{R,R}^{2n}. 
\] 
Thus the same link \(L\) is represented by the \(T^{2n}\)-link
\[
T^{2n}((2,2),(R,1-nR),(R;0)).
\]
By Theorem~\ref{Thm:TnLinksLorenzLikeTemplates}, this gives a
representative of \(L\) in the Lorenz-like template
\(\mathscr L(0,2n)\). By Corollary~\ref{Cor:AssociatedLorenzLinkTn}, its associated Lorenz link is
\[
L_R^+=T((2,2),(R,1-nR)).
\]

Since \(n<0\), we have
\[
1-nR\geq R+1.
\]
Consequently, the standard positive \(R\)-braid representing \(L_R^+\)
contains a full twist on all \(R\) strands. Therefore, by Franks and
Williams \cite[Corollary 2.4]{Franks-Williams:BraidsAndJonesPolynomial}, we have that
\[
\operatorname{br}(L_R^+)=R.
\]
By Lemma~\ref{lem:trip-number-associated-lorenz}, the corresponding
Lorenz-like representative of \(L\) has trip number
\[
\tau_{\mathscr L(0,2n)}(L)
=
\operatorname{br}(L_R^+)
=
R.
\]

Since \(R\) can be chosen arbitrarily large, the fixed link \(L\) admits
Lorenz-like representatives with arbitrarily large trip number.
\end{proof}

\begin{corollary}\label{Cor:VolumeBoundT2nLink}
Let \(n\in\mathbb Z\), and let
\[
L=
T^{2n}((r_1,s_1),\ldots,(r_k,s_k),(r_{k+1};d))
\]
be a \(T^{2n}\)-link in \(S^3\). Let \(\operatorname{Vol}(L)\) denote the
sum of the volumes of the hyperbolic pieces of \(S^3\setminus L\). If
\[
\beta=
\operatorname{br}\bigl(
T((r_1,s_1),\ldots,(r_k,s_k))
\bigr),
\]
then
\[
\operatorname{Vol}(L)
\leq
12v_{\textup{tet}}(\beta^2+3\beta).
\]
\end{corollary}

\begin{proof}
By Theorem~\ref{Thm:TnLinksLorenzLikeTemplates}, the link \(L\) admits the
corresponding representative in the Lorenz-like template
\(\mathscr L(0,2n)\). Let \(\tau\) denote the trip number of this
representative. By Theorem~\ref{Thm:UpperBdTripNumber},
\[
\operatorname{Vol}(L)
\leq
12v_{\textup{tet}}(\tau^2+3\tau).
\]

By Corollary~\ref{Cor:AssociatedLorenzLinkTn}, applied with \(2n\) in place
of \(n\), the Lorenz link associated to this representative is
\[
L^+
=
T((r_1,s_1),\ldots,(r_k,s_k)).
\]
Lemma~\ref{lem:trip-number-associated-lorenz} therefore gives
\[
\tau
=
\operatorname{br}(L^+)
=
\operatorname{br}\bigl(
T((r_1,s_1),\ldots,(r_k,s_k))
\bigr)
=
\beta.
\]
Substituting \(\tau=\beta\) into the preceding inequality yields
\[
\operatorname{Vol}(L)
\leq
12v_{\textup{tet}}(\beta^2+3\beta). \qedhere
\]
\end{proof}

\begin{theorem}\label{Thm:VolumeBoundAssociatedLorenzLink}
Let
\[
L=
T((r_1,s_1),\ldots,(r_k,s_k),(r_{k+1},s_{k+1}),d)
\]
be a generalised T-link, and suppose that
\[
s_{k+1}\geq nr_{k+1}
\]
for some integer \(n\). Let \(\operatorname{Vol}(L)\) denote the sum of
the volumes of the hyperbolic pieces of \(S^3\setminus L\). Define
\[
L_n^+
=
T((r_1,s_1),\ldots,(r_k,s_k),
(r_{k+1},s_{k+1}-nr_{k+1})),
\]
where the final block is omitted if
\(s_{k+1}-nr_{k+1}=0\). Let
\[
\beta=\operatorname{br}(L_n^+).
\]
Then
\[
\operatorname{Vol}(L)
\leq
12v_{\textup{tet}}(\beta^2+3\beta).
\]
\end{theorem}

\begin{proof}
By Theorem~\ref{Thm:GeneralisedTLinksT2nLinksTemplates}, the link \(L\) is
equivalent to the \(T^{2n}\)-link
\[
L'
=
T^{2n}((r_1,s_1),\ldots,(r_k,s_k),
(r_{k+1},s_{k+1}-nr_{k+1}),(r_{k+1};d)),
\]
where the block
\[
(r_{k+1},s_{k+1}-nr_{k+1})
\]
is omitted if \(s_{k+1}-nr_{k+1}=0\).

By Theorem~\ref{Thm:TnLinksLorenzLikeTemplates}, the link \(L'\) admits a
representative
\[
\widetilde L\subset\mathscr L(0,2n).
\]
Moreover, by replacing $n$ with $2n$ in Corollary~\ref{Cor:AssociatedLorenzLinkTn}, we have that the Lorenz link associated to \(\widetilde L\) is precisely \(L_n^+\).

Let
\[
\tau=\tau_{\mathscr L(0,2n)}(\widetilde L).
\]
Since \(L\), \(L'\), and \(\widetilde L\) represent the same link type,
Theorem~\ref{Thm:UpperBdTripNumber} gives
\[
\operatorname{Vol}(L)
=
\operatorname{Vol}(L')
\leq
12v_{\textup{tet}}(\tau^2+3\tau).
\]
By Lemma~\ref{lem:trip-number-associated-lorenz},
\[
\tau
=
\operatorname{br}(L_n^+)
=
\beta.
\]
Substituting \(\tau=\beta\) into the preceding inequality yields
\[
\operatorname{Vol}(L)
\leq
12v_{\textup{tet}}(\beta^2+3\beta). \qedhere
\]
\end{proof}

\begin{definition}\label{Def:LorenzBraidIndex}
Let \(L\subset S^3\) be a link. Consider all generalised T-link
representations of \(L\),
\[
L=
T((r_1,s_1),\ldots,(r_k,s_k),(r_{k+1},s_{k+1}),d),
\]
and all integers \(n\) satisfying
\[
s_{k+1}\geq n r_{k+1}.
\]
For each such pair of choices, let
\[
L^+_n
=
T((r_1,s_1),\ldots,(r_k,s_k),
(r_{k+1},s_{k+1}-nr_{k+1}))
\]
be the associated Lorenz link, where the final block is omitted if
\(s_{k+1}-nr_{k+1}=0\). The \emph{Lorenz braid index} of \(L\), denoted by
\(\operatorname{lbi}(L)\), is defined by
\[
\operatorname{lbi}(L)
=
\min \operatorname{br}(L^+_n),
\]
where the minimum is taken over all generalised T-link representations
of \(L\) and all integers \(n\) satisfying \(s_{k+1}\geq n r_{k+1}\).
\end{definition}

Since every link in \(S^3\) admits a generalised T-link
representation, and since one may choose an integer \(n\) satisfying
\(s_{k+1}\geq nr_{k+1}\), the set
\[
\left\{
\operatorname{br}(L_n^+)
\right\}
\]
appearing in Definition~\ref{Def:LorenzBraidIndex} is nonempty. Its
elements are non-negative integers. Hence, by the well-ordering principle,
this set has a least element, so \(\operatorname{lbi}(L)\) is
well-defined. Since the minimum is taken over all generalised T-link
representations of \(L\), the resulting value depends only on the ambient
isotopy class of \(L\).

The Lorenz braid index is the braid-theoretic counterpart of the minimal
trip number introduced in Definition~\ref{Def:MinTripNumber}. The two
invariants agree. 

\begin{proposition}\label{Prop:MinTripNumberEqualsLBI}
For every link \(L\subset S^3\),
\[
\widetilde n(L)=\operatorname{lbi}(L),
\]
where \(\widetilde n(L)\) denotes the minimal trip number of \(L\).
\end{proposition} 

\begin{proof}
We first prove that
\[
\operatorname{lbi}(L)\leq \widetilde n(L).
\]
Let
\[
L'\subset \mathscr L(0,2n)
\]
be a Lorenz-like representative of \(L\). By
Theorem~\ref{Thm:TnLinksLorenzLikeTemplates}, the link \(L'\) is represented
by a \(T^{2n}\)-link.

Suppose first that this \(T^{2n}\)-link is not one of the %exceptional
trivial links \(T((2,1),d)\) arising when \(n>0\). Then it
determines a generalised T-link representation
\[
L=
T((r_1,s_1),\ldots,(r_k,s_k),
(r_{k+1},s_{k+1}),d)
\]
satisfying
\[
s_{k+1}\geq nr_{k+1}.
\]
Moreover, this representation is equivalent to
\[
T^{2n}((r_1,s_1),\ldots,(r_k,s_k),
(r_{k+1},s_{k+1}-nr_{k+1}),(r_{k+1};d)).
\]
By Corollary~\ref{Cor:AssociatedLorenzLinkTn}, applied with \(2n\) in
place of \(n\), the associated Lorenz link is
\[
L_n^+
=
T((r_1,s_1),\ldots,(r_k,s_k),
(r_{k+1},s_{k+1}-nr_{k+1})).
\]
Lemma~\ref{lem:trip-number-associated-lorenz} therefore gives
\[
\tau_{\mathscr L(0,2n)}(L')
=
\operatorname{br}(L_n^+).
\]
Since \(L_n^+\) is one of the associated Lorenz links considered in the
definition of \(\operatorname{lbi}(L)\), we obtain
\[
\operatorname{lbi}(L)
\leq
\tau_{\mathscr L(0,2n)}(L').
\]
Taking the minimum over all Lorenz-like representatives \(L'\) gives
\[
\operatorname{lbi}(L)\leq \widetilde n(L).
\]

Conversely, choose a generalised T-link representation
\[
L=
T((r_1,s_1),\ldots,(r_k,s_k),
(r_{k+1},s_{k+1}),d)
\]
and an integer \(n\) satisfying
\[
s_{k+1}\geq nr_{k+1}
\]
such that
\[
\operatorname{lbi}(L)=\operatorname{br}(L_n^+).
\]
By Theorem~\ref{Thm:GeneralisedTLinksT2nLinksTemplates}, this generalised
T-link is equivalent to the \(T^{2n}\)-link
\[
\widetilde L = T^{2n}((r_1,s_1),\ldots,(r_k,s_k),
(r_{k+1},s_{k+1}-nr_{k+1}),(r_{k+1};d)).
\]
Theorem~\ref{Thm:TnLinksLorenzLikeTemplates} gives a representative
\[
\widetilde L\subset\mathscr L(0,2n).
\]
By Corollary~\ref{Cor:AssociatedLorenzLinkTn}, its associated Lorenz link
is precisely \(L_n^+\). Hence
\[
\widetilde n(L)
\leq
\tau_{\mathscr L(0,2n)}(\widetilde L)
=
\operatorname{br}(L_n^+)
=
\operatorname{lbi}(L),
\]
where the equality in the middle follows from
Lemma~\ref{lem:trip-number-associated-lorenz}.

Combining the two inequalities gives
\[
\widetilde n(L)=\operatorname{lbi}(L).
\]

It remains to consider the trivial links \(T((2,1),d)\) arising when
\(n>0\). The link \(T((2,1),d)\) is the unlink with \(d+1\) components.
It admits a representative in \(\mathscr L(0,2n)\) consisting of
\(d+1\) boundary-parallel trivial components contained in the left strip
of the template. None of these components travels from one strip to the
other. Hence this representative has trip number zero, and therefore
\[
\widetilde n\bigl(T((2,1),d)\bigr)=0.
\]

Since this representative does not meet the right strip, removing the
\(2n\) half twists produces no nontrivial Lorenz component. Thus its
associated Lorenz link is the empty link. Adopting the convention
\[
\operatorname{br}(\varnothing)=0,
\]
we obtain
\[
\operatorname{lbi}\bigl(T((2,1),d)\bigr)=0.
\]
Consequently,
\[
\widetilde n\bigl(T((2,1),d)\bigr)
=
\operatorname{lbi}\bigl(T((2,1),d)\bigr)
=
0. \qedhere
\]
\end{proof}

\begin{theorem}\label{Thm:VolumeLorenzBraidIndex}
Let \(L\subset S^3\) be a link, and let \(\operatorname{Vol}(L)\) denote the
sum of the volumes of the hyperbolic pieces in the JSJ decomposition of
\(S^3\setminus L\). Then
\[
\operatorname{Vol}(L)
\leq
12v_{\mathrm{tet}}
\bigl(
\operatorname{lbi}(L)^2
+
3\operatorname{lbi}(L)
\bigr),
\]
where \(v_{\mathrm{tet}}\) denotes the volume of the regular ideal
tetrahedron.
\end{theorem}

\begin{proof}
By Theorem~\ref{Thm:AllLinksGeneralisedTLinks} and the definition of the Lorenz braid index, there exists a generalised T-link representation of $L$,
\[
L=
T((r'_1,s'_1),\ldots,(r'_\ell,s'_\ell),
(r'_{\ell+1},s'_{\ell+1}),d'),
\]
and an integer \(n\) satisfying
\[
s'_{\ell+1}\geq n r'_{\ell+1}
\]
such that the associated Lorenz link
\[
L_n^+
=
T((r'_1,s'_1),\ldots,(r'_\ell,s'_\ell),
(r'_{\ell+1},s'_{\ell+1}-nr'_{\ell+1}))
\]
satisfies
\[
\operatorname{br}(L_n^+)=\operatorname{lbi}(L).
\]
Here the final block is omitted if
\(s'_{\ell+1}-nr'_{\ell+1}=0\).

Applying Theorem~\ref{Thm:VolumeBoundAssociatedLorenzLink} to this
generalised T-link representation gives
\[
\operatorname{Vol}(L)
\leq
12v_{\textup{tet}}
\bigl(
\operatorname{br}(L_n^+)^2
+
3\operatorname{br}(L_n^+)
\bigr).
\]
Since
\[
\operatorname{br}(L_n^+)=\operatorname{lbi}(L),
\]
we conclude that
\[
\operatorname{Vol}(L)
\leq
12v_{\textup{tet}}
\bigl(
\operatorname{lbi}(L)^2
+
3\operatorname{lbi}(L)
\bigr). \qedhere
\]
\end{proof}

\section{Bounded volume and unbounded classical complexity}
\label{Sec:BoundedVolumeUnboundedBraid}
In this section, we show that the Lorenz braid index has a better ability to capture geometric
information in some settings compared to the classical topological invariants. 

More precisely, we construct an explicit family of hyperbolic Lorenz knots
whose classical braid index and Seifert genus both tend to infinity, while
their Lorenz braid indices remain uniformly bounded. By \refthm{VolumeLorenzBraidIndex}, 
the volumes of these knot complements are therefore uniformly
bounded.

The construction is based on adding many full twists on a growing number of
strands. These full twists increase the classical braid index and the genus
of the resulting knots. However, from the point of view of Lorenz-like
templates, this twisting can be removed to recover an associated Lorenz
link whose braid index is independent of the amount of twisting. Thus the
Lorenz braid index records the bounded parent-link complexity underlying
the construction, rather than the large twisting complexity visible in a
classical braid representative.

The main technical point is to prove that the knots in the family are
hyperbolic. We first recall a hyperbolicity criterion for positive
generalised \(T\)-knots and a braid-index computation for the associated
positive part. We then use these facts, together with results on full
twists and generalised cabling, to rule out the torus and satellite cases
for the twisted family. This produces hyperbolic Lorenz knots with bounded
Lorenz braid index but unbounded classical braid index and unbounded Seifert
genus.
 
The following lemma is precisely the hyperbolicity criterion proved in
\cite[Theorem~7.10]{de2024lorenz}, using the improvement of
\cite[Corollary~1.3]{twofulltwists} which allows the case \(k\geq 1\).

\begin{lemma}\label{lem:positive-part-hyperbolic}
Let \(p,q\) be positive coprime integers, and suppose that
\[
p>q\geq r_n>\cdots>r_1>1.
\]
Then, for every \(k\geq1\), the generalised \(T\)-knot
\[
T((r_1,s_1),\ldots,(r_n,s_n),(p,kp+q))
\]
is hyperbolic.
\end{lemma}

\begin{lemma}\label{lem:braid-index-positive-part}
Let \(p,q\) be positive integers, and suppose that
\[
p>r_n>\cdots>r_1>1.
\]
Then, for every \(k\geq1\), the generalised \(T\)-knot
\[
L_k=
T((r_1,s_1),\ldots,(r_n,s_n),(p,kp+q))
\]
has braid index equal to \(p\).
\end{lemma}

\begin{proof}
The knot \(L_k\) is the closure of the positive braid
\[
B_k=
(\sigma_1\cdots\sigma_{r_1-1})^{s_1}\cdots
(\sigma_1\cdots\sigma_{r_n-1})^{s_n}
(\sigma_1\cdots\sigma_{p-1})^{kp+q}
\]
on \(p\) strands. Hence
\[
\operatorname{br}(L_k)\leq p.
\]

Since \(k\geq1\), the final factor contains at least one full twist on all
\(p\) strands:
\[
(\sigma_1\cdots\sigma_{p-1})^{kp+q}
=
\bigl((\sigma_1\cdots\sigma_{p-1})^p\bigr)^k
(\sigma_1\cdots\sigma_{p-1})^q.
\]
Thus \(B_k\) is a positive \(p\)-braid containing a full twist on all
\(p\) strands. By the theorem of Franks and Williams
\cite[Corollary~2.4]{Franks-Williams:BraidsAndJonesPolynomial}, this braid
realises the braid index of its closure. Therefore
\[
\operatorname{br}(L_k)=p. \qedhere
\]
\end{proof}

\begin{definition}
By a \emph{generalized $q$--cabling} of a link $L$ we mean a link $L'$
contained in the interior of a tubular neighbourhood $L\times D^2$ of $L$ such that:
\begin{enumerate}
\item each fiber $D^2$ intersects $L'$ transversely in $q$ points; and
\item all strands of $L'$ are oriented in the same direction as $L$ itself.
\end{enumerate}
\end{definition}

The following theorem, due to Williams~\cite{Williams}, states that braid
index is multiplicative under generalised cabling.

\begin{theorem}\label{Williams}
Let \(L\) be a link whose components are all non-trivial knots, and let
\(L'\) be a generalised \(q\)-cabling of \(L\). Then
\[
\beta(L')=q\,\beta(L),
\]
where \(\beta(\cdot)\) denotes braid index.
\end{theorem}

The following lemma was proved in \cite[Lemma~2.3]{twofulltwists}.

\begin{lemma}\label{Lemma10}
Let \(L'\) be a generalized \(q\)-cabling of the unknot \(L\), where \(L\) is represented by a positive braid with \(n\) strands, and assume \(n > 1\). Suppose further that \(\partial N(L)\) is disjoint from the braid axis of \(L\). Then \(L'\) can be represented as a positive braid with \(s - 1\) full twists, where \(n = s q\). In particular, \(L'\) has braid index equal to \(q\).
\end{lemma}

\begin{theorem}\label{hyperbolic}
Suppose that
\[
p'>p>q\geq r_n>\cdots>r_1>1,
\]
with
\[
\gcd(p,q)=1
\qquad\text{and}\qquad
\gcd(p,p')=1.
\]
Then, for every \(k,k'\geq1\), the generalised \(T\)-knot
\[
T((r_1,s_1),\ldots,(r_n,s_n),
(p,kp+q-1),(p',k'p'+1))
\]
is hyperbolic.
\end{theorem}

\begin{proof}
Let
\[
A=
(\sigma_1\cdots\sigma_{r_1-1})^{s_1}\cdots
(\sigma_1\cdots\sigma_{r_n-1})^{s_n}.
\]
Then the knot in the statement is the closure of the positive braid
\[
B=
A(\sigma_1\cdots\sigma_{p-1})^{kp+q-1}
(\sigma_1\cdots\sigma_{p'-1})^{k'p'+1}
\]
on \(p'\) strands.

We first observe what happens after removing the \(k'\) full twists on the
\(p'\) strands. Let
\[
B'=
A(\sigma_1\cdots\sigma_{p-1})^{kp+q-1}
(\sigma_1\cdots\sigma_{p'-1}).
\]
The closure of \(B'\) is isotopic to
\[
T((r_1,s_1),\ldots,(r_n,s_n),(p,kp+q)).
\]
Indeed, the final factor
\[
\sigma_1\cdots\sigma_{p'-1}
\]
allows us to destabilise successively the rightmost \(p'-p\) strands.
After these destabilisations, the remaining braid is
\[
A(\sigma_1\cdots\sigma_{p-1})^{kp+q}.
\]
Therefore, by Lemma~\ref{lem:positive-part-hyperbolic}, the closure of
\(B'\) is hyperbolic. Moreover, by
Lemma~\ref{lem:braid-index-positive-part}, its braid index is equal to
\(p\).

We now prove that the closure \(K=\widehat B\) is hyperbolic. Since \(B\)
is a positive braid containing a full twist on all \(p'\) strands, the
theorem of Franks and Williams
\cite[Corollary~2.4]{Franks-Williams:BraidsAndJonesPolynomial} gives
\[
\operatorname{br}(K)=p'.
\]
In particular, \(K\) is non-trivial.

Suppose first that \(K\) is a torus knot. Since \(K\) has braid index
\(p'\), it is a \((p',m)\)-torus knot for some \(m\). The braid \(B\)
contains \(k'\) full twists on the \(p'\) strands. By Los' theorem
\cite[Corollary~1.2]{Los}, any two minimal braid representatives of the
same torus knot with the same braid axis are isotopic in the complement
of the axis. Therefore, after performing surgery on the braid axis which
removes these \(k'\) full twists, the resulting knot must again be a torus
knot.

However, the resulting knot is the closure of \(B'\), which is isotopic to
\[
T((r_1,s_1),\ldots,(r_n,s_n),(p,kp+q)).
\]
This knot is hyperbolic by
Lemma~\ref{lem:positive-part-hyperbolic}, and hence cannot be a torus knot.
This is a contradiction. Thus \(K\) is not a torus knot.

It remains to rule out the satellite case. Suppose that \(K\) is a
satellite knot, and let \(T\) be an essential torus in
\(S^3\setminus K\). Since \(B\) is positive and contains \(k'\) positive
full twists on all \(p'\) strands, the torus \(T\) can be isotoped to be
disjoint from the braid axis \cite[Theorem~5.3 and Lemma~3.1]{Ito2025}. Moreover, the core of the solid torus
bounded by \(T\) is represented by a positive braid with \(k'\) full
twists with respect to the same braid axis as \(B\) \cite[Theorem~4.1]{Satellite}.

Thus \(K\) is a generalised \(b\)-cabling of the core of the solid torus
bounded by \(T\), for some \(b>1\). In particular,
\[
b\mid p'.
\]

Now perform surgery on the braid axis which removes the \(k'\) full twists
on the \(p'\) strands. The knot \(K\) becomes the closure of \(B'\), and
the torus \(T\) becomes a torus \(T'\) disjoint from the braid axis of
\(B'\).

If \(T'\) is knotted, then Williams' theorem gives
\[
\operatorname{br}(\widehat{B'})
=
b\,\operatorname{br}(C'),
\]
where \(C'\) denotes the core of the solid torus bounded by \(T'\).
If \(T'\) is unknotted, then
Lemma~\ref{Lemma10} gives
\[
\operatorname{br}(\widehat{B'})=b.
\]
In either case,
\[
b\mid\operatorname{br}(\widehat{B'}).
\]

By Lemma~\ref{lem:braid-index-positive-part},
\[
\operatorname{br}(\widehat{B'})=p.
\]
Therefore
\[
b\mid p.
\]
Since also \(b\mid p'\) and \(\gcd(p,p')=1\), we obtain
\[
b=1,
\]
contradicting \(b>1\). Hence \(K\) is not a satellite knot.

We have shown that \(K\) is non-trivial, not a torus knot, and not a
satellite knot. By Thurston's hyperbolisation theorem for knots in
\(S^3\), the knot \(K\) is hyperbolic.
\end{proof}

\begin{lemma}\label{lem:associated-lorenz-fixed}
Let \(q,s_1,\ldots,s_n\) be positive integers, and suppose that
\[
p'>p>r_n>\cdots>r_1>1.
\]
For \(k,k'\geq1\), let
\[
K_{k,k'}
=
T((r_1,s_1),\ldots,(r_n,s_n),
(p,kp+q-1),(p',k'p'+1)).
\]
Then the Lorenz link obtained by removing the \(k'\) positive full twists
from the final block is independent of \(k'\). More precisely, this
associated Lorenz link is
\[
L_k
=
T((r_1,s_1),\ldots,(r_n,s_n),(p,kp+q)).
\]
In particular,
\[
\operatorname{br}(L_k)=p,
\]
and hence the braid index of the associated Lorenz link is independent of
\(k'\).
\end{lemma}

\begin{proof}
Set
\[
A=
(\sigma_1\cdots\sigma_{r_1-1})^{s_1}\cdots
(\sigma_1\cdots\sigma_{r_n-1})^{s_n}.
\]
Then \(K_{k,k'}\) is the closure of the positive braid
\[
B_{k,k'}
=
A
(\sigma_1\cdots\sigma_{p-1})^{kp+q-1}
(\sigma_1\cdots\sigma_{p'-1})^{k'p'+1}
\]
on \(p'\) strands.

We may write the final factor as
\[
(\sigma_1\cdots\sigma_{p'-1})^{k'p'+1}
=
\left(
(\sigma_1\cdots\sigma_{p'-1})^{p'}
\right)^{k'}
(\sigma_1\cdots\sigma_{p'-1}).
\]
The factor
\[
\left(
(\sigma_1\cdots\sigma_{p'-1})^{p'}
\right)^{k'}
\]
consists of \(k'\) positive full twists on all \(p'\) strands. Removing
these full twists therefore gives the positive braid
\[
B_k'
=
A
(\sigma_1\cdots\sigma_{p-1})^{kp+q-1}
(\sigma_1\cdots\sigma_{p'-1}).
\]

The closure of \(B_k'\) admits \(p'-p\) successive positive Markov
destabilisations, removing the rightmost strands \(p',p'-1,\ldots,p+1\).
After these destabilisations, the final factor
\[
\sigma_1\cdots\sigma_{p'-1}
\]
reduces to
\[
\sigma_1\cdots\sigma_{p-1}.
\]
Hence the resulting braid on \(p\) strands is
\[
A
(\sigma_1\cdots\sigma_{p-1})^{kp+q-1}
(\sigma_1\cdots\sigma_{p-1})
=
A
(\sigma_1\cdots\sigma_{p-1})^{kp+q},
\]
whose closure is
\[
T((r_1,s_1),\ldots,(r_n,s_n),(p,kp+q))
=
L_k.
\]

Therefore, the associated Lorenz link depends on \(k\) and on the fixed
initial parameters, but not on \(k'\). Finally, by
Lemma~\ref{lem:braid-index-positive-part},
\[
\operatorname{br}(L_k)=p. \qedhere
\]
\end{proof}

\begin{corollary}\label{cor:uniform-volume-bound}
Using the assumptions of Theorem~\ref{hyperbolic}, for \(k,k'\geq1\), let
\[
K_{k,k'}
=
T((r_1,s_1),\ldots,(r_n,s_n),
(p,kp+q-1),(p',k'p'+1)).
\]
Then
\[
\operatorname{Vol}(K_{k,k'})
\leq
12v_{\textup{tet}}(p^2+3p)
\]
for every \(k,k'\geq1\). In particular, for each fixed \(k\), the volumes
of the knots \(K_{k,k'}\), with \(k'\geq1\), are uniformly bounded.
Indeed, the bound is independent of both \(k\) and \(k'\).
\end{corollary}

\begin{proof}
By Lemma~\ref{lem:associated-lorenz-fixed}, the Lorenz link associated
with the given generalised T-link representation of \(K_{k,k'}\) is
\[
L_k
=
T((r_1,s_1),\ldots,(r_n,s_n),(p,kp+q)).
\]
By Lemma~\ref{lem:braid-index-positive-part},
\[
\operatorname{br}(L_k)=p.
\]
Therefore, Theorem~\ref{Thm:VolumeBoundAssociatedLorenzLink} gives
\[
\begin{aligned}
\operatorname{Vol}(K_{k,k'})
&\leq
12v_{\textup{tet}}
\bigl(
\operatorname{br}(L_k)^2
+
3\operatorname{br}(L_k)
\bigr) \\
&=
12v_{\textup{tet}}(p^2+3p).
\end{aligned}
\]
This bound is independent of both \(k\) and \(k'\).
\end{proof}

\begin{lemma}\label{lem:genus-positive-braid-knot}
Let \(K\) be a knot represented as the closure of a positive braid
\(\beta\) on \(p\) strands with \(c\) crossings. Then
\[
g(K)=\frac{c-p+1}{2}.
\]
\end{lemma}

\begin{proof}
Apply Seifert's algorithm to the standard closed braid diagram of
\(\beta\). The resulting canonical Seifert surface \(F\) is obtained from
\(p\) disks, corresponding to the \(p\) Seifert circles, by attaching
\(c\) positively twisted bands, one for each crossing. Therefore,
\[
\chi(F)=p-c.
\]

Since \(\partial F=K\) is a knot, the surface \(F\) is connected and has
one boundary component. Hence
\[
\chi(F)=1-2g(F).
\]
It follows that
\[
1-2g(F)=p-c,
\]
and therefore
\[
g(F)=\frac{c-p+1}{2}.
\]

It remains to show that \(F\) is genus-minimising. Since every positive
braid is homogeneous, Stallings' theorem~\cite{Genus} implies that the
closure of \(\beta\) is fibered and that its canonical Seifert surface
\(F\) is a fiber surface. A fiber surface of a knot is genus-minimising~\cite[Theorem 2.4 (Neuwirth-Stallings)]{Stoimenow:MinimalGenus}. 

Consequently,
\[
g(K)=g(F)=\frac{c-p+1}{2}. \qedhere
\]
\end{proof}

\begin{theorem}\label{thm:bounded-volume-unbounded-braid}
Let \(s_1,\ldots,s_n\) be positive integers, and let \(p,q\) be positive
coprime integers satisfying
\[
p>q\geq r_n>\cdots>r_1>1.
\]
Fix \(k\geq1\). Let \(\{p'_j\}_{j\geq1}\) be a sequence of positive
integers such that
\[
p'_j>p,
\qquad
\gcd(p,p'_j)=1,
\qquad\text{and}\qquad
p'_j\longrightarrow\infty.
\]
For each \(j\), choose \(k'_j\geq1\), and define
\[
K_j=
T((r_1,s_1),\ldots,(r_n,s_n),
(p,kp+q-1),(p'_j,k'_jp'_j+1)).
\]
Then each \(K_j\) is a hyperbolic Lorenz knot. Moreover,
\[
\operatorname{br}(K_j)=p'_j,
\qquad
\operatorname{lbi}(K_j)\leq p,
\]
and
\[
g(K_j)\geq
\frac{k'_jp'_j(p'_j-1)}{2}.
\]
Consequently,
\[
\operatorname{br}(K_j)\longrightarrow\infty
\qquad\text{and}\qquad
g(K_j)\longrightarrow\infty,
\]
whereas
\[
\operatorname{Vol}(K_j)
\leq
12v_{\mathrm{tet}}(p^2+3p)
\]
for every \(j\). In particular, the hyperbolic volumes of the knots
\(K_j\) are uniformly bounded.
\end{theorem} 

\begin{proof}
By Theorem~\ref{hyperbolic}, each \(K_j\) is hyperbolic. Moreover, all
the exponents in its generalised T-link representation are positive,
so \(K_j\) is a T-link. Since T-links and Lorenz links coincide,
each \(K_j\) is a Lorenz knot \cite{Birman-Kofman:NewTwistOnLorenzLinks}.

We first bound the Lorenz braid index. By
Lemma~\ref{lem:associated-lorenz-fixed}, the Lorenz link associated with
the given generalised T-link representation of \(K_j\), obtained by
removing the \(k'_j\) full twists on the \(p'_j\) strands, is
\[
L_k=
T((r_1,s_1),\ldots,(r_n,s_n),(p,kp+q)).
\]
This link is independent of \(j\). By
Lemma~\ref{lem:braid-index-positive-part},
\[
\operatorname{br}(L_k)=p.
\]
Therefore,
\[
\operatorname{lbi}(K_j)
\leq
\operatorname{br}(L_k)
=
p
\]
for every \(j\).

The volume bound in terms of the Lorenz braid index now gives
\[
\begin{aligned}
\operatorname{Vol}(K_j)
&\leq
12v_{\mathrm{tet}}
\bigl(
\operatorname{lbi}(K_j)^2+
3\operatorname{lbi}(K_j)
\bigr)\\
&\leq
12v_{\mathrm{tet}}(p^2+3p).
\end{aligned}
\]
Thus the volumes are bounded above by a constant independent of \(j\).

The knot \(K_j\) is represented by a positive braid on \(p'_j\) strands
containing \(k'_j\geq1\) full twists on all \(p'_j\) strands. Hence, by
\cite[Corollary~2.4]{Franks-Williams:BraidsAndJonesPolynomial},
\[
\operatorname{br}(K_j)=p'_j.
\]
Since \(p'_j\to\infty\), it follows that
\[
\operatorname{br}(K_j)\to\infty.
\]

Finally, let \(c_j\) be the number of crossings in the positive braid
defining \(K_j\). By
Lemma~\ref{lem:genus-positive-braid-knot},
\[
g(K_j)=\frac{c_j-p'_j+1}{2}.
\]
The final block
\[
(\sigma_1\cdots\sigma_{p'_j-1})^{k'_jp'_j+1}
\]
alone contributes
\[
(k'_jp'_j+1)(p'_j-1)
\]
crossings. Therefore,
\[
\begin{aligned}
g(K_j)
&\geq
\frac{(k'_jp'_j+1)(p'_j-1)-p'_j+1}{2}\\
&=
\frac{k'_jp'_j(p'_j-1)}{2}.
\end{aligned}
\]
Since \(k'_j\geq1\) and \(p'_j\to\infty\), we conclude that
\[
g(K_j)\to\infty.  \qedhere
\]
\end{proof}
 
Theorem~\ref{thm:bounded-volume-unbounded-braid} provides an explicit family
of Lorenz knots exhibiting the phenomenon studied by Purcell and
Zupan~\cite{PurcellZupan}: classical topological complexity can grow
arbitrarily large without forcing the hyperbolic volume to grow. In this
family, both the classical braid index and the Seifert genus tend to
infinity, whereas, for fixed \(k\), the associated Lorenz link remains
unchanged and has braid index \(p\). Consequently, the Lorenz braid indices
remain uniformly bounded, and so do the hyperbolic volumes. Thus, the
Lorenz braid index captures a form of geometric complexity relevant to
volume which is not reflected by the classical braid index or the Seifert
genus.

\bibliographystyle{amsplain}  
\bibliography{biblio_MathSciNet.bib} 
\end{document}